\documentclass[11pt]{article}

\usepackage[pdftex]{graphicx}

\usepackage{amssymb}

\usepackage{amsthm}

\makeatletter
\renewcommand\section{\@startsection{section}{1}{\z@}%
                                  {-3.5ex \@plus -1ex \@minus -.2ex}%
                                  {2.3ex \@plus.2ex}%
                                  {\normalfont\large\bfseries}}
\makeatother

\DeclareGraphicsExtensions{.jpg}
\begin{document}

\title{Proof of Lov\'{a}sz conjecture for odd order}

\author{Misa Nakanishi \thanks{E-mail address : nakanishi@mx-keio.net}}
\date{}
\maketitle

\begin{abstract}
Lov\'{a}sz conjectured that every connected vertex-transitive graph contains a hamilton path in 1970. First we reveal the structure of connected vertex-transitive graphs with an odd number of vertices. Then we prove that every connected vertex-transitive graph with an odd number of vertices is hamiltonian. \\
keywords : hamilton cycle, vertex-transitive graph, odd order
\end{abstract}

\newtheorem{thm}{Theorem}[section]
\newtheorem{lem}{Lemma}[section]
\newtheorem{prop}{Proposition}[section]
\newtheorem{cor}{Corollary}[section]
\newtheorem{rem}{Remark}[section]
\newtheorem{conj}{Conjecture}[section]
\newtheorem{claim}{Claim}[section]
\newtheorem{fact}{Fact}[section]
\newtheorem{obs}{Observation}[section]

\newtheorem{defn}{Definition}[section]
\newtheorem{propa}{Proposition}
\renewcommand{\thepropa}{\Alph{propa}}
\newtheorem{conja}[propa]{Conjecture}

\section{Introduction}
In this paper, we reveal the structure of connected vertex-transitive graphs with an odd number of vertices from a graph theoretic view and prove that the Lov\'{a}sz conjecture for odd order is true.

Regarding a vertex-transitive graph, its structure has been the focus of many previous studies. In 1977, Leighton extended the notion of a circulant to vertex-transitive graphs \cite{Leighton}. In 1982, Alspach and Parsons showed how to construct a vertex-transtive graph from transitive permutation groups \cite{Alspach}. In 1994 and 1996, McKay and Praeger investigated vertex-transitve graphs that are not Cayley graphs \cite{McKay, McKayII}. In 2022, Georgakopoulos and Wendland showed how to generalize the construction of Cayley graphs to represent vertex-transitive graphs \cite{Georgakopoulos}.

On the other hand, in 1970, Lov\'{a}sz conjectured that every connected vertex-transitive graph contains a hamilton path. This is called the Lov\'{a}sz conjecture and has been studied by many authors, but remains open. In particular, research and surveys on hamilton paths have been published so far with respect to Cayley graphs \cite{Pak, Lanel}. In addition, in 2011, Christofides, Hladk\'{y}, and M\'{a}th\'{e} proved that every sufficiently large dense connected vertex-transitive graph is hamiltonian \cite{Christofides}. In 2012, Kutnar, Maru\u{s}i\u{c}, and Zhang showed that every connected vertex-transitive graph of order $10p$ ($p \ne 7$, a prime) has a hamilton path \cite{Kutnar}.

From the definition of a vertex-transitve graph, we focus on its structure and show that a connected vertex-transitive graph with an odd number of vertices has a 2-factor which consists of an odd number of odd cycles with the same length. By this structure, we prove that a connected vertex-transitive graph with an odd number of vertices is hamiltonian.

\section{Notation and property}
In this paper, a graph $G := (V, E)$ is finite, undirected, and simple with the vertex set $V$ and edge set $E$. We follow the notations presented in \cite{Diestel}. A vertex-transitive graph is a graph $G$ in which, given any
two vertices $v_1$ and $v_2$ in $G$, there is some automorphism
\[f \colon G \to G\]
such that
\[f(v_1) = v_2.\]
A graph is vertex-transitive if and only if its graph complement is.

\begin{conja}[\cite{Lovasz}]\label{Lovasz}
Every connected vertex-transitive graph contains a hamilton path.
\end{conja}

\section{Proof of Lov\'{a}sz conjecture for odd order}

\begin{lem}\label{L}
A connected vertex-transitive graph with an odd number of vertices has a 2-factor which consists of an odd number of odd cycles with the same length.
\end{lem}

\begin{proof}
Let $G$ be a connected vertex-transitive graph with an odd number of vertices. If $G$ is a complete graph or 2-regular graph, this statement obviously holds. Suppose otherwise. Since $G$ is vertex-transitive, for $v_1, v_2, \cdots, v_k \in V(G)$ ($k \geq 2$), there is some automorphism $f \colon G \to G$ such that $f(v_i) = v_{i + 1}$ for $1 \leq i \leq k - 1$ and $f(v_k) = v_1$. Let $A \colon G \to G$ be an automorphism of $G$ as follows: (a) for some $x \in V(G)$, $A(x) \ne x$, and (b) for any two vertices $v_1, v_2 \in V(G)$ such that $A(v_1) = v_{2}$, $A$ takes minimum $k$ such that $A(v_i) = v_{i + 1}$ for $1 \leq i \leq k - 1$ and $A(v_k) = v_1$. Let $\mathcal{A}$ be the set of all $A$. Let $\mathcal{Q}(A)$ be the family of all minimal sets of vertices that map one to the other in some $A \in \mathcal{A}$. Let $\mathcal{R}(A)$ be the family of all minimal sets of odd vertices that map one to the other in some $A \in \mathcal{A}$. Now, $\mathcal{R}(A) \subseteq \mathcal{Q}(A)$. Since $|G|$ is odd, $\mathcal{R}(A) \ne \emptyset$ and $|\mathcal{R}(A)|$ is odd. Since the complement of $G$ is also vertex-transitive, for any $Q \in \mathcal{Q}(A)$, if $|Q| \ne 1$, then $G[Q]$ or the complement of $G[Q]$ has a factor which consists of prime paths or prime cycles. Hence, by its minimality, we can consider $\mathcal{Q}(A)$ as the family of all minimal sets of prime vertices that map one to the other in $A$, except for trivial vertices. For some $Q_1, Q_2 \in \mathcal{Q}(A)$, suppose $|Q_1| \ne |Q_2|$ and two vertices taken one from each $Q_1$ and $Q_2$ map one to the other in some $A^{*} \in \mathcal{A}$. For $Q \in \mathcal{Q}(A^{*})$ that contains these two vertices, by its minimality, each vertex in $Q$ is taken from a distinct set in $\mathcal{Q}(A)$, which contradicts that $G$ is regular. Thus, for all $Q_1, Q_2 \in \mathcal{Q}(A)$, $|Q_1| = |Q_2|$, and $\mathcal{R}(A) = \mathcal{Q}(A)$. Now, for all $R \in \mathcal{R}(A)$, $|R| \ne 1$. Since $|\mathcal{R}(A)|$ is odd, from its symmetry, (a) for all $R \in \mathcal{R}(A)$, $G[R]$ has a prime cycle with the length $|R|$, or (b) for all $R \in \mathcal{R}(A)$, $G[R]$ has no edges. For the case (b), by equally dividing $\mathcal{R}(A)$ into $i$ sets ($i \geq 3$, odd), we can easily take $i$ odd cycles with the same length. The proof is completed.   
\end{proof}

\begin{thm}\label{A2}
A connected vertex-transitive graph with an odd number of vertices is hamiltonian.
\end{thm}

\begin{proof}
Let $G$ be a connected vertex-transitive graph with an odd number of vertices. Let $H_0 = G$. If $H_0$ is 2-regular then $G$ is hamiltonian. Suppose that $H_0$ is not 2-regular. $|H_0|$ is odd and $H_0$ is a connected vertex-transitive graph, thus $H_0$ has a 2-factor which consists of an odd number of odd cycles with the same length by Lemma \ref{L}. Note that if $H_0$ has one spanning cycle, then $G$ is hamiltonian. Suppose otherwise. Contract each of these cycles into a vertex. Let $H_1$ be the resulting graph. $|H_1|$ is odd and $H_1$ is a connected vertex-transitive graph, thus $H_1$ has a 2-factor which consists of an odd number of odd cycles with the same length by Lemma \ref{L}. Now, $H_1$ has a hamilton cycle. By reversing these operations from $H_1$, we find a hamilton cycle in $G$.
\end{proof}


\begin{thebibliography}{}


\bibitem{Alspach}
B. Alspach and T. D. Parsons: A construction for vertex-transitive graphs. Can. J. Math., Vol. XXXIV, No. 2 (1982) 307-318

\bibitem{Christofides}
D. Christofides, J. Hladk\'{y}, and A. M\'{a}th\'{e}: A proof of the dense version of Lov\'{a}sz conjecture. Electronic Notes in Discrete Mathematics 38 (2011) 285-290

\bibitem{Diestel}
R. Diestel: Graph Theory Fourth Edition. Springer (2010)

\bibitem{Georgakopoulos}
A. Georgakopoulos and A. Wendland: Presentations for vertex-transitive graphs. Journal of Algebraic Combinatorics 55 (2022) 795–826

\bibitem{Kutnar}
K. Kutnar, D. Maru\u{s}i\u{c}, and C. Zhang: Hamilton paths in vertex-transitive graphs of order $10p$. European Journal of Combinatorics 33 (2012) 1043–1077

\bibitem{Lanel}
G. H. J. Lanel, H. K. Pallage, J. K. Ratnayake, S. Thevasha, and B. A. K. Welihinda: A survey on Hamiltonicity in Cayley graphs and digraphs on different groups. Discrete Mathematics, Algorithms and Applications 11:5 (2019) 1930002

\bibitem{Leighton}
F. T. Leighton: Circulants and the characterization of vertex-transitive graphs. J Res Natl Bur Stand (1977). 1983 Nov-Dec; 88(6): 395–402

\bibitem{Lovasz}
L. Lov\'{a}sz: Combinatorial structures and their applications Proc. Calgary Internat. Conf. Calgary, Alberta, 1969, Gordon and Breach, New York (1970) 243-246 Problem 11

\bibitem{McKay}
B. D. McKay and C. E. Praeger: Vertex-transitive graphs which are not Cayley graphs, I. J. Austral. Math. Soc. (Series A) 56 (1994) 53-63

\bibitem{McKayII} 
B. D. McKay and C. E. Praeger: Vertex-transitive graphs that are not Cayley graphs, II. Journal of Graph Theory 22:4 (1996) 321-334

\bibitem{Pak}
I. Pak and R. Radoi\u{c}i\'{c}: Hamiltonian paths in Cayley graphs. Discrete Mathematics 309 (2009) 5501–5508





\end{thebibliography}
\end{document}